\documentclass[10pt]{amsart}
\usepackage{amstext}
\usepackage{amsfonts}
\usepackage{amssymb}
\usepackage{amsbsy}
\usepackage{amsmath}
\usepackage{latexsym}
\usepackage{xy}
\usepackage{hhline}
\xyoption{all}
\newcommand{\vtx}[1]{*+[o][F-]{\scriptscriptstyle #1}} %the command to draw a vertex

\newcounter{num}[section] %

\newenvironment{theo}
{\refstepcounter{num}%
\bigskip\noindent{\bf Theorem~\arabic{section}.\arabic{num}. }\it}
\newenvironment{cor}
{\refstepcounter{num}%
\bigskip\noindent{\bf Corollary~\arabic{section}.\arabic{num}. }\it}
\newenvironment{lemma}
{\refstepcounter{num}%
\bigskip\noindent{\bf Lemma~\arabic{section}.\arabic{num}. }\it}
\newenvironment{remark}
{\refstepcounter{num}%
\bigskip\noindent{\bf Remark~\arabic{section}.\arabic{num}.}}

\newcommand{\Ref}[1]{(\ref{#1})}

\newcounter{thepic}

\newenvironment{eq}{\begin{equation}}{\end{equation}}

\newcommand{\si}{\sigma}
\newcommand{\al}{\alpha}
\newcommand{\be}{\beta}
\newcommand{\ga}{\gamma}

\newcommand{\de}{\delta}

\newcommand{\LA}{\langle}
\newcommand{\RA}{\rangle}

\newcommand{\ov}[1]{\overline{#1}}
\newcommand{\un}[1]{{\underline{#1}} }

\newcommand{\tr}{\mathop{\rm tr}}

\newcommand{\mdeg}{\mathop{\rm mdeg}}
\newcommand{\Char}{\mathop{\rm char}}

\newcommand{\Sp}{S\!p}

\newcommand{\FF}{{\mathbb{F}}}   % base field
\newcommand{\NN}{{\mathbb{N}}}
\newcommand{\algA}{\mathcal{A}}    %algebra A
\newcommand{\algB}{\mathcal{B}}    %algebra B
\newcommand{\X}{\LA X\RA}
\newcommand{\EX}{\LA \widetilde{X}\RA}
\newcommand{\Q}{\mathcal{Q}}    %Quiver Q
\newcommand{\loopR}[3]{%
\begin{picture}(20,0)(#1,#2)
\put(-2,1){\llap{$\scriptstyle #3$}} \put(11,3){\circle{20}} \put(20,6){\vector(1,-4){1}}
\end{picture}}
\newcommand{\loopL}[3]{%
\begin{picture}(20,0)(#1,#2)
\put(22,1){$\scriptstyle #3$} \put(9,3){\circle{20}} \put(0,6){\vector(-1,-4){1}}
\end{picture}}

\newcommand{\PhiLarge}[1]{\widetilde{\Phi}_{#1}}  % the map from $\AlgLarge$ to $R^{GL(n)}$ 
\newcommand{\PhiAbs}[1]{\widehat{\Phi}_{#1}}  % the map from the absolutely free algebra $\si\X$ to $R^{GL(n)}$ 
\begin{document}
\renewcommand{\refname}{References}
\thispagestyle{empty}

\title{Identities for matrix invariants of the symplectic group}%
\author{{Artem A. Lopatin}}%
\noindent\address{\noindent{}Artem A. Lopatin  
\newline\hphantom{iiii} Sobolev Institute of Mathematics, 
\newline\hphantom{iiii} Siberian Branch of the Russian Academy of Sciences 
\newline\hphantom{iiii} (IM SBRAS). 
\newline\hphantom{iiii} Pevtsova street, 13,
\newline\hphantom{iiii} 644043, Omsk, Russia.
\newline\hphantom{iiii} http://www.iitam.omsk.net.ru/\~{}lopatin}%
\email{artem\underline{ }lopatin@yahoo.com}%

\vspace{1cm}
\maketitle {\small
\begin{quote}
\noindent{\sc Abstract. } The general linear group acts on the space of several linear maps on the vector space as the basis change. Similarly, we have the actions of the orthogonal and symplectic groups. Generators and identities for  the corresponding polynomial invariants over a characteristic zero field were described by Sibirskii, Procesi and Razmyslov in 1970s. In 1992 Donkin started to transfer these results to the case of infinite fields of arbitrary characteristic. We completed this transference for fields of odd characteristic  by establishing identities for the symplectic matrix invariants over infinite fields of odd characteristic.     
\medskip


\end{quote}
}

%=======================================================================================
%=======================================================================================
%---Sec1-------------------------------------------------------------------------------
\section{Introduction}\label{section_intro}

We work over an infinite field  $\FF$ of arbitrary characteristic $p=\Char{\FF}$. All vector spaces, algebras and modules are over $\FF$ and all algebras are associative with unity unless otherwise stated.

Consider a group $G$ from the list $GL(n)$, $O(n)=\{A\in \FF^{n\times n}\,|\,A A^T =E\}$, $SO(n)=\{A\in O(n)\,|\,\det(A)=1\}$,  $\Sp(n)=\{A\in \FF^{n\times n}\,|\,AA^{\ast}=E\}$,  where we assume that $p\neq2$ in case $G\in\{O(n),SO(n)\}$ and $n$ is even in case $G=\Sp(n)$. Here $\FF^{n\times n}$ is the space of $n\times n$ matrices over $\FF$ and $A^{\ast}=-J A^T J$ is the {\it symplectic} transpose of $A$, where   %
$J=J_n=\left(
\begin{array}{cc}
0& E \\
-E& 0\\
\end{array}
\right)$ is the matrix of the non-degenerate skew-symmetric bilinear form. The group $G$ acts on the space $V=(\FF^{n\times n})^{\oplus d}$ by the diagonal conjugation: 
$$g\cdot (A_1,\ldots,A_d)=(gA_1g^{-1},\ldots,gA_dg^{-1})$$
for $g\in G$ and $A_1,\ldots,A_d$ in $\FF^{n\times n}$. The coordinate ring of $V$, i.e. the algebra of all polynomial maps $V\to \FF$, is the polynomial ring 
$$R=\FF[x_{ij}(k)\,|\,1\leq i,j\leq n,\, 1\leq k\leq d]$$ 
in $n^2 d$ variables, where $x_{ij}(k)$ sends $(A_1,\ldots,A_d)$ to the $(i,j)^{\rm th}$ entry of $A_k$. 
The algebra of {\it matrix $G$-invariants} is the set of all polynomial maps $f\in R$ that are constants on $G$-orbits of $V$, i.e., $f(g\cdot v)=f(v)$ for all $g\in G$ and $v\in V$. We denote this algebra by $R^G$.

%-add-% Let $G$ be $O(n)$ or $SO(n)$. Then $V_{+}=S_{+}^{\oplus d}$ and $V_{-}=S_{-}^{\oplus d}$ are $G$-submodules of $V$, where $S_{+}$ ($S_{-}$, respectively) stands for the space of $n\times n$ symmetric (skew-symmetric, respectively) matrices over $\FF$. The corresponding coordinate rings are $R_{+}=\FF[x_{ij}(k)\,|\,1\leq i\leq j\leq n,\, 1\leq k\leq d]$ and $R_{-}=\FF[x_{ij}(k)\,|\,1\leq i< j\leq n,\, 1\leq k\leq d]$. The algebra $R^G_{+}$ ($R^G_{-}$, respectively) of $G$-invariants of symmetric (skew-symmetric, respectively) matrices is defined similarly to $R^G$. 

For $f\in R$ denote by $\deg{f}$ its {\it degree} and by $\mdeg{f}$ its {\it multidegree}, i.e., $\mdeg{f}=(t_1,\ldots,t_d)$, where $t_k$ is the total degree of the polynomial $f$ in $x_{ij}(k)$, $1\leq i,j\leq n$, and $\deg{f}=t_1+\cdots+t_d$.  By the Hilbert--Nagata Theorem on invariants, each of the considered algebras of invariants $R^G$ is a finitely generated 
algebra. The algebra $R^G$ also have $\NN_0$-grading by degrees and $\NN_0^d$-grading by multidegrees, where $\NN_0$ stands for non-negative integers. Denote by $D(R^G)$ the maximal degree of elements of a minimal (by inclusion) $\NN^d$-homogeneous set of generators (m.h.s.g.)~for $R^G$. 
%-add-% One can see that $D(I)$ is well-defined (see~... for the detail). 

Generating sets for the considered algebras of invariants are known (see Section~\ref{section_generators}) as well as relations between generators for $R^{GL(n)}$ and $R^{O(n)}$ (see Section~\ref{section_identities}). In case $p=0$ relations between generators for $R^{\Sp(n)}$ are also known (see~\cite{Procesi76}).  The key difference between the case of zero and positive characteristic is the following property obtained in~\cite{DKZ02}:
\begin{eq}\label{eq_weyl}
D(R^{GL(n)})\to\infty \text{ as }d\to \infty \text{ if and only if }0<p\leq n.
\end{eq}

In this paper we describe the ideal of relations for $R^{Sp(n)}$ over a field of odd characteristic (see Theorem~\ref{theo_relationsSp} below). As a corollary, working under assumption that $p\neq2$ we extend the property~\Ref{eq_weyl} to the cases of $R^{O(n)}$ and $R^{\Sp(n)}$ (see Theorem~\ref{theo_weyl}).

%-add-%
%\begin{enumerate}
% \item[$\bullet$] the ideals of relations for $R^{SO(n)}$, $R_{+}^{O(n)}$ and $R_{-}^{O(n)}$ are described;

% \item[$\bullet$] it is established that property~\Ref{eq_weyl} holds for $R^{O(n)}$ and $R_{+}^{O(n)}$, but does not hold for $R_{-}^{O(n)}$;

% \item[$\bullet$] a m.h.s.g.~is found for $R^{O(4)}_{-}$ in case $d=2$;

% \item[$\bullet$] in case $p=0$ or $p>n$ a m.h.s.g.~for $R^{GL(n)}$ is described in terms of a basis of the relatively free finitely generated algebra $N_{n,d}$ with the identity $x^n=0$ (see~...); 

% \item[$\bullet$] in case $p=0$ or $p>n$ a m.h.s.g.~for $R^{O(n)}$ is described in terms of a basis of the relatively free finitely generated algebra $A_{n,d}$ with the identities $x^n=0$ and ... (see~... for the details).
%\end{enumerate}
  
Throughout this paper we write $\NN_0$ for non-negative integers and $\NN$ for positive integers. Given $\un{t}\in\NN_0^v$, we denote $t_1+\cdots+t_u$ by $|\un{t}|$.

%=======================================================================================
%=======================================================================================
%---S2-------------------------------------------------------------------------------
\section{Generators}\label{section_generators}

To formulate the result describing generators of the algebra $R^G$, we introduce the following notations. The ring $R$ is generated by the entries of $n\times n$ {\it generic} matrices $X_k=(x_{ij}(k))_{1\leq i,j\leq n}$ ($1\leq k\leq d$). 

%-add-% Similarly, define $n\times n$ {\it symmetric generic} matrices $X_k^{+}$ and $n\times n$ {\it skew-symmetric generic} matrices $X_k^{-}$, i.e., these matrices have the following entries:
%$$(X_k^{+})_{ij} = \left\{\begin{array}{rl}
%x_{ij}(k),& \text{if }  i\leq j\\
%x_{ji}(k),& \text{if }  i>j\\
%\end{array}
%\right. \quad \text{ and }\quad 
%(X_k^{-})_{ij} =
%\left\{\begin{array}{rl}
%x_{ij}(k),& \text{if }  i<j\\
%-x_{ji}(k),& \text{if }  i>j\\
%0,& \text{otherwise}\\
%\end{array}
%\right..
%$$

Consider an arbitrary $n\times n$ matrix $A=(a_{ij})$ over some commutative ring. Denote coefficients in the characteristic polynomial of $A$ by $\sigma_t(A)$, i.e., %
$$\det(\lambda E - A)=\sum_{t=0}^{n} (-1)^t\lambda^{n-t}\sigma_t(A).$$
So, $\sigma_0(A)=1$, $\sigma_1(A)=\tr(A)$ and $\sigma_n(A)=\det(A)$. 
 
%-add-% Assume that $n$ is even. Define the {\it generalized pfaffian} by $\ov{\pf}(A)=\pf(A-A^T)$, where $\pf$ stands for the pfaffian of a skew-symmetric matrix. For $\FF=\QQ$ we have the following formula
%\begin{eq}\label{eq_pf}
%\ov{\pf}(A)=\frac{1}{(n/2)!}\sum\limits_{\pi\in S_n}(-1)^{\pi}
%\prod\limits_{i=1}^{n/2} a_{\pi(2i-1),\pi(2i)},
%\end{eq}%
%where $(-1)^{\pi}$ stands for the sign of $\pi$. Note that for an $n\times n$ skew-symmetric matrix $B=(b_{ij})$ the equality $\ov{\pf}(B)=2^{n/2}\pf(B)$ holds.

%-add-% For $n\times n$ matrices $A_1=(a_{ij}(1)),\ldots,A_r=(a_{ij}(s))$ and positive integers $t_1,\ldots,t_r$, satisfying $t_1+\cdots+t_r=n/2$, consider the polynomial $\ov{\pf}(a_1 A_1+\cdots+a_r A_r)$ in the variables $a_1,\ldots,a_r$. The {\it partial linearization} $\ov{\pf}_{t_1,\ldots,t_r}(A_1,\ldots,A_r)$ of the generalzed pfaffian is the coefficient of $a_1^{t_1}\cdots a_r^{t_r}$ in this polynomial.

Part~(a) of the following theorem was proven by Donkin~\cite{Donkin92a}, parts~(b), (c) by Zubkov~\cite{Zubkov99}.

%--Th2.1---------------------------------------------------------------
\begin{theo}\label{theo_intro1} The algebra of matrix $G$-invariants $R^G$ is generated by the following elements:
\begin{enumerate}
 \item[(a)] $\si_t(A)$ ($1\leq t\leq n$ and $A$ ranges over all monomials in $X_1,\ldots,X_d$), if $G=GL(n)$; 

 \item[(b)] $\si_t(B)$ ($1\leq t\leq n$ and $B$ ranges over all monomials in $X_1,\ldots,X_d$, $X_1^T,\ldots,X_d^T)$, if $G=O(n)$; 

 \item[(c)] $\si_t(C)$ ($1\leq t\leq n$ and $C$ ranges over all monomials in $X_1,\ldots,X_d$, $X_1^{\ast},\ldots,X_d^{\ast}$), if $G=\Sp(n)$.
 \end{enumerate}
\end{theo}
\bigskip

We can assume that in the formulation of Theorem~\ref{theo_intro1} each of monomials $A,B,C$ is {\it primitive}, i.e., is not equal to a power of a shorter monomial. If $p=0$ or $p>n$, then in Theorem~\ref{theo_intro1} it is enough to take traces $\tr(U)$, where $U$ can be non-primitive, instead of $\si_t(U)$ in order to generate the algebra $R^G$. The corresponding results were obtained earlier than Theorem~\ref{theo_intro1} by Sibirskii~\cite{Sibirskii68} and Procesi~\cite{Procesi76}.

%-add-% The next statement follows from the main result of paper~\cite{Lopatin09JAlg} by Lopatin.

%-add-% \begin{cor}\label{corr_intro2} The algebra $I$ form the list %$R_{+}^{O(n)}$, $R_{-}^{O(n)}$, $R_{+}^{SO(n)}$, $R_{-}^{SO(n)}$ is generated %by the following elements:
%\begin{enumerate}
% \item[$\bullet$] $\si_t(B)$ ($1\leq t\leq t$), if $G=O(n)$ or $G=SO(n)$ for odd $n$; 
%
% \item[$\bullet$] $\si_t(B)$, $\ov{\pf}_{t_1,\ldots,t_r}(B_1,\ldots,B_r)$ ($1\leq t\leq t$), if $G=SO(n)$ and $n$ is even.
%\end{enumerate}
%
%Here matrices $B,B_1,\ldots,B_r$ are monomials in $X_1^{+},\ldots,X_{d}^{+}$ in case $I=R_{+}^{O(n)}$ or $I=R_{+}^{SO(n)}$ and they are monomials in $X_1^{-},\ldots,X_{d}^{-}$ in case $I=R_{-}^{O(n)}$ or $I=R_{-}^{SO(n)}$. 
%\end{cor}
%\bigskip

%---R2.2---------------------------------------------------------------
\begin{remark}
It is not difficult to see that elements from Theorem~\ref{theo_intro1} are in fact invariants. Namely, the action of $G$ on $V$ induces the action on $R$ as follows: $g\cdot x_{ij}(k)$ is the $(i,j)$-th entry of $g^{-1} X_k g$. Given $f\in R$, we have $f\in R^G$ if and only if $g\cdot f=f$ for all $g\in G$. Then the following properties give us the required: $\si_t(g A g^{-1})=\si_t(A)$.
\end{remark}

%=======================================================================================
%=======================================================================================
%---S3-------------------------------------------------------------------------------
\section{Identities}\label{section_identities}

To describe relations for $R^G$ we should represent it as a quotiont of the corresponding free algebra $\si\EX$ by the ideal of relations $K_n$. All necessary definitions are given below. 

Let $\X$ be the semigroup (without unity) freely generated by {\it letters}  
\begin{enumerate}
\item[$\bullet$] $x_1,\ldots,x_d$, if $G=GL(n)$;

\item[$\bullet$] $x_1,\ldots,x_d,x_1^T,\ldots,x_d^T$, otherwise.
\end{enumerate} 
Denote $\X^{\#}=\X\sqcup\{1\}$, i.e., we endow $\X$ with the unity.  Assume that $a=a_1\cdots a_r$ and $b$ are elements of $\X$, where $a_1,\ldots,a_r$ are letters. 
\begin{enumerate}
\item[$\bullet$] Introduce the involution ${}^T$ on $\X$ as follows. If $G=GL(n)$, then $a^T=a$. Otherwise, we set $b^{TT}=b$ for a letter $b$ and $a^T=a_r^T\cdots a_1^T\in\X$. 

\item[$\bullet$] We say that $a$ and $b$ are {\it cyclic equivalent} and write $a\stackrel{c}{\sim} b$ if $a=a_1a_2$ and $b=a_2a_1$ for some $a_1,a_2\in\X^{\#}$.  If $a\stackrel{c}{\sim} b$ or $a\stackrel{c}{\sim} b^T$, then we say that $a$ and $b$ are {\it equivalent} and write $a\sim b$.
\end{enumerate}
An element from $\X$ is called {\it primitive} if it is not equal to a power of a shorter monomial.  Obviously, if $a\sim b$ for a primitive $a\in\X$, then $b$ is also primitive. Introduce the natural lexicographical linear order on $\X$ by setting $x_1>x_1^T>x_2>x_2^T>\cdots$ and $ab>a$ for $a,b\in\X$. (Note that we can actually consider any other lexicographical linear order). 
\begin{enumerate}
\item[$\bullet$] Let $\EX\subset\X$ be a subset of maximal (with respect to the introduced lexicographical order on $\X$) representatives of $\stackrel{}{\sim}$-equivalence classes of primitive elements.

\item[$\bullet$] Let $\si\EX$ be the ring with unity of commutative polynomials over $\FF$ freely generated by ``symbolic'' elements $\si_t(a)$, where $t>0$ and $a\in\EX$. 
\end{enumerate}

For a letter $b\in\X$ define
$$X_{b}=
\left\{
\begin{array}{rl}
X_{k},&\text{if } b=x_k\\
X_{k}^T,&\text{if } b=x_k^T \text{ and }G=O(n)\\
X_{k}^{\ast},&\text{if } b=x_k^T \text{ and }G=\Sp(n)\\
\end{array}
\right..
$$
Given $a=a_1\cdots a_r\in\X$, where $a_i$ is a letter, we set $X_{a}=X_{a_1}\cdots X_{a_r}$. Consider the surjective homomorphism of algebras
$$\PhiLarge{n}:\si\EX \to R^G$$ %
defined by $\si_t(a) \to \si_t(X_a)$, if $t\leq n$, and $\si_t(a) \to 0$ otherwise. Note that for all $n\times n$ matrices $A,B$ over $R$ and $1\leq t\leq n$ we have $\si_t(A^{\de})=\si_t(A)$, $(A^{\de})^{\de}=A$, and $(AB)^{\de}=B^{\de}A^{\de}$, where $\de$ stands for the transposition or symplectic transposition. Hence the map $\PhiLarge{n}$ is well defined. Its kernel $K_{n}$ is the ideal of {\it relations} for $R^G$.

To define elements generating $K_n$ we need the algebra  $\si\X$. 
\begin{enumerate}
\item[$\bullet$] Let $\FF\X$ and $\FF\X^{\#}$ be the free associative algebras (without and with unity, respectively) with the $\FF$-bases $\X$ and $\X^{\#}$, respectively. Note that elements of $\FF\X$ and $\FF\X^{\#}$ are {\it finite} linear combinations of monomials from $\X$ and $\X^{\#}$, respectively.

\item[$\bullet$] Let $\si\X$ be a ring with unity of commutative polynomials over $\FF$ freely generated by ``symbolic'' elements $\si_t(a)$, where $t\geq1$ and $a$ ranges over polynomials from $\FF\X$ with coefficient $1$ in the highest term with respect to the introduced lexicographical order on $\X$. Define 
$$\si_t(\al a)=\al^t\si_t(a)$$
for $\al\in\FF$ and denote $\si_t(0)=0$.
\end{enumerate}

We will use the following conventions for elements of $\si\EX$ and $\si\X$: $\si_0(a)=1$ and $\si_1(a)=\tr(a)$. Similarly to $\PhiLarge{n}:\si\EX\to R^G$,  we define the surjective homomorphism of algebras $\PhiAbs{n}:\si\X\to R^G$.

The following lemma implies that any element of $\si\X$ can be considered as an element of $\si\EX$. In what follows, we apply this lemma without a reference to it. See Section~2 of~\cite{Lopatin13JPAA} for the definitions of $F_t$ and $P_{t,l}$.

%-----L3.1--------------------------------------------------------
\begin{lemma}\label{lemma_L}
We have $\si\EX\simeq \si\X/L$ for the ideal $L$ generated by
\begin{enumerate}
\item[(a)] $\si_t(a_1+\cdots+a_s)=F_t(a_1,\ldots,a_s)$, 

\item[(b)] $\si_t(a^l)=P_{t,l}(a)$, 

\item[(c)] $\si_t(bc)=\si_t(cb)$, 

\item[(d)] $\si_t(a^T)=\si_t(a)$ in case $G$ is $O(n)$ or $\Sp(n)$, 
\end{enumerate}
where $t>0$, $l,s>1$, $a_1,\ldots,a_s\in\FF\X$, $a,b,c\in\X$. Moreover,  $L$ belongs to the kernel of $\PhiAbs{n}$. 
\end{lemma}
\bigskip

The isomorphism from Lemma~\ref{lemma_L} was established by Donkin (see~\cite{Donkin93Camb}) in case $G=GL(n)$ and by Lopatin (see Lemma~3.1 of~\cite{Lopatin12ART}) in case $G=O(n)$. The proof for the case of $G=\Sp(n)$ follows immediately from the case of $G=O(n)$.  The second statement of the lemma was proven in Remark~3.6 of~\cite{Lopatin13JPAA}.

Assume that $G$ is $O(n)$ or $\Sp(n)$. Let us recall the definition of element $\si_{t,r}(a,b,c)$ of $\si\X$, where $t,r\geq0$ and $a,b,c\in\FF\X$. For short, we set $x=x_1$, $y=x_2$, and $z=x_3$. Consider the quiver (i.e., the oriented graph) $\Q$:
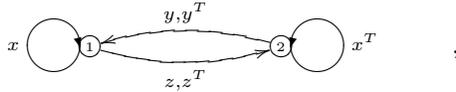
$$
\loopR{0}{0}{x} %
\xymatrix@C=1cm@R=1cm{ %
\vtx{1}\ar@1@/^/@{<-}[rr]^{y,y^T} &&\vtx{2}\ar@1@/^/@{<-}[ll]^{z,z^T}\\
}%
\loopL{0}{0}{x^T}\qquad\qquad,
$$
where there are two arrows from vertex $2$ to vertex $1$ as well as from $1$ to $2$. By abuse of notation arrows of $\Q$ are denoted by letters from $\X$. For an arrow $a$ denote by $a'$ its head and by $a''$ its tail. A sequence of arrows $a_1\cdots a_s$ of $\Q$ is a {\it path} of $\Q$ if $a_i''=a_{i+1}'$ for all $1\leq i< s$. The head of the path $a$ is $a'=a_1'$ and the tail is $a''=a_s''$. A path $a$ is {\it closed} if $a'=a''$. Denote the multidegree of a monomial $a$ in arrows of $\Q$ by $\mdeg(a)=(\deg_{x}(a) + \deg_{x^T}(a),\deg_{y}(a) + \deg_{y^T}(a),\deg_{z}(a) + \deg_{z^T}(a))$. We set 
$$\si_{t,r}(x,y,z)=\sum (-1)^{\xi} \si_{k_1}(e_1)\cdots \si_{k_q}(e_q),$$
where the sum ranges over all closed paths $e_1,\ldots,e_q$ in $\Q$ that are pairwise different with respect to $\sim$-equivalence and  $k_1,\ldots,k_q>0$ ($q>0$) satisfying  $k_1\mdeg(e_1)+\cdots+k_q\mdeg(e_q)=(t,r,r)$. Here $\xi=t+\sum_{i=1}^q k_i(\deg_{y}{e_i}+\deg_{z}{e_i}+1)$. Given $a,b,c\in\FF\X$ we define $\si_{t,r}(a,b,c)$ as the result of the substitutions $x\to a$, $y\to b$, $z\to c$ in $\si_{t,r}(x,y,z)$.

Part~(a) of the following theorem was proven by Zubkov~\cite{Zubkov96} and part~(b) by Lopatin~\cite{Lopatin12ART},~\cite{Lopatin12JPAA}:

%----Th3.2----------------------------------------------
\begin{theo}\label{theo_relationsGLO} The ideal of relations $K_{n}$ for $R^G\simeq  \si\EX/K_{n}$ is generated by 
\begin{enumerate}
\item[(a)] $\si_t(a)$ for $t>n$, if $G=GL(n)$;

\item[(b)] $\sigma_{t,r}(a,b,c)$ for $t+2r>n$ ($t,r\geq0$), if $G=O(n)$ and  $p\neq2$.
\end{enumerate}
Here $a,b,c$ ranges over $\FF\X$.
\end{theo}
\bigskip

Define the element $\varrho_{t,r}(x,y,z)$ of $\si\X$ by
$$\varrho_{t,r}(x,y,z)=\sum (-1)^{t+k_1+\cdots+k_q} \si_{k_1}(e_1)\cdots \si_{k_q}(e_q),$$
where $e_1,\ldots,e_q$, $k_1,\ldots,k_q$ are the same as in the definition of $\si_{t,r}$.

%----Th3.3----------------------------------
\begin{theo}\label{theo_relationsSp} Assume $p\neq2$ and $G=\Sp(n)$. Then the ideal of relations $K_{n}$ for $R^{\Sp(n)}\simeq  \si\EX/K_{n}$ is generated by $\varrho_{t,r}(a,b,c)$ for $t+2r>n$ ($t,r\geq0$), where $a,b\in\FF\X$ and $c\in\FF\X^{\#}$.
\end{theo}
\bigskip

The key difference of the relations in case $G=O(n)$ and in case $G=\Sp(n)$ is that in the first case the degree of any non-zero relation (from $\si\EX$) is greater than $n$, but in the second case there are non-trivial relations of degree $\frac{n}{2}+1$. The proof of the theorem is given at the end of Section~\ref{section_quivers}.

%----R3.4--------------------------------------
\begin{remark}\label{remark1} Denote by $\FF_p\subset \FF$ the field of characteristic $p$, generated by $1$. Note that generators of $R^{\Sp(n)}$ as well as elements from the formulation of Theorem~\ref{theo_relationsSp} are defined over $\FF_p$. Hence the standard linear algebra arguments imply that without loss of generality in the proof of Theorem~\ref{theo_relationsSp} we can assume that $\FF$ is algebraically closed. 
\end{remark}
\begin{proof}
This remark follows from the fact that if  $\algA\simeq\algB / K$ for algebras $\algA$ and $\algB$ over $\FF_p$ and an ideal $K$ of $\algB$, then $\algA\otimes\FF\simeq\algB\otimes\FF / K\otimes\FF$, where the tensor product is over $\FF_p$.
\end{proof}

%=======================================================================================
%=======================================================================================
%---S4-------------------------------------------------------------------------------
\section{Isomorphism of algebras}\label{section_isomorphism}

In this section we assume that $\FF$ is an arbitrary field and $n$ is even. Let us recall that over a ring of characteristic two a matrix $A$ is called skew-symmetric if $A=-A^T$ and every diagonal element of $A$ is zero. To define a subalgebra $I_n$ of $R\otimes \FF[y_{ij}\,|\,1\leq i<j\leq n]$ we denote by $Y$ the $n\times n$ skew-symmetric matrix with the $(i,j)^{\rm th}$ entry equal to $y_{ij}$ for $i<j$.  The algebra $I_n$ is generated by all $\si_t(A_1 Y\cdots A_r Y)$ for $1\leq t\leq n$, $r>0$, where  $A_i\in\{X_1,\ldots,X_d,X_1^T,\ldots,X_d^T\}$ for all $i$. Consider the homomorphism of algebras
$$\Psi_n:R\otimes \FF[y_{ij}\,|\,1\leq i<j\leq n]\to R$$ 
defined by $\Psi_n(X_k)=X_kJ$ ($1\leq k\leq d$) and $\Psi_n(Y)=-J$. Here $\Psi_n(X_k)$ stands for the $n\times n$ matrix such that its $(i,j)^{\rm th}$ entry is $\Psi_n(x_{ij}(k))$ and $\Psi_n(Y)$ is defined similarly. In what follows, we use similar notations.

%----L4.1-------------------------------------------------
\begin{lemma}\label{lemma_iso}
The restriction of $\Psi_n$ to $I_n$ is an isomorphism of algebras $I_n$ and $R^{\Sp(n)}$.
\end{lemma}
\bigskip
This lemma is proven at the end of this section. Note that $\Psi_n(I_n)\subset R^{\Sp(n)}$. 

Let $I'_n$ be the  subalgebra of $R\otimes \FF[z_{ij}\,|\,1\leq i,j\leq n]$ generated by elements  $\si_t(A_1 Z J Z^T\cdots A_r Z J Z^T)$ for $1\leq t\leq n$, $r>0$, where  $A_i\in\{X_1,\ldots,X_d,X_1^T,\ldots,X_d^T\}$ for all $i$ and $Z=(z_{ij})_{1\leq i,j\leq n}$.

%----L4.2-------------------------------------------------
\begin{lemma}\label{lemma_hom}
The exists a unique homomorphism of algebras $\theta_n: I'_n\to I_n$ that sends $\si_t(A_1 Z J Z^T\cdots A_r Z J Z^T)$ to $\si_t(A_1 Y\cdots A_r Y)$ for all $A_i\in\{X_1,\ldots,X_d$, $X_1^T,\ldots,X_d^T\}$ and $1\leq t\leq n$.
\end{lemma}
\begin{proof} Given a monomial $a=A_1 Z J Z^T\cdots A_r Z J Z^T$, we write $\theta_n(a)$ for the monomial $A_1 Y\cdots A_r Y$. Let $f=\sum_i\al_i \si_{t_{i1}}(a_{i1})\cdots\si_{t_{ir_i}}(a_{ir_i})$ be an element of $I'_n$, where $\al_i\in\FF$. Denote by $h$ the element $\sum_i\al_i \si_{t_{i1}}(\theta_n(a_{i1}))\cdots\si_{t_{ir_i}}(\theta_n(a_{ir_i}))$ of $I_n$. To prove the lemma, it is enough to show that if $f=0$, then $h=0$.

Assume that $f=0$. Then the result of substitution $Z\to B$ in $f$ is zero for every $B\in\FF^{n\times n}$. It is well-known that for any skew-symmetric bilinear form $f$ on $\FF^n$ we have $\FF^n=V_0\oplus V_1$, where $V_0$ is the kernel of $f$ and the restriction of $f$ to $V_1$ is given by the matrix $J_{\dim V_1}$ with respect to some basis of $V_1$ (for example, see~\cite{Lang65}). Therefore, for any skew-symmetric $n\times n$ matrix $C$ over $\FF$ there is a  $B\in \FF^{n\times n}$ such that $BJB^T=C$. Thus, 
the result of substitution $Y\to C$ in $h$ is zero for every skew-symmetric matrix $C\in\FF^{n\times n}$. Since $\FF$ is infinite, the last condition implies that $h=0$. 
\end{proof}

Now we can proof Lemma~\ref{lemma_iso}.

\begin{proof}
Define the homomorphism of algebras 
$$\mu_n:R\to R\otimes \FF[z_{ij}\,|\,1\leq i,j\leq n]$$
by $\mu_n(X_k)= Z^T X_k Z J$, $1\leq k\leq n$. Note that $\mu_n(R^{\Sp(n)})\subset I'_n$.
Let us introduce the following notations. Given an $n\times n$ matrix $A$ and $\al\in\{0,1\}$, we set
$$A^{\al}=\left\{
\begin{array}{rl}
A,&\text{if } \al=0\\
A^T,&\text{if } \al=1\\ 
\end{array}\right.,\;\;
A^{\ov{\al}}=\left\{
\begin{array}{rl}
A,&\text{if } \al=0\\
A^{\ast},&\text{if } \al=1\\ 
\end{array}\right..$$ 
Assume that $\al,\al_i\in\{0,1\}$ and $A_i\in\{X_1,\ldots,X_d\}$ for $1\leq i\leq r$. Definitions of $\mu_n$ and $\Psi_n$ imply that
$$\mu_n(X_k^{\ov{\al}})=(-1)^{\al}Z^TX_k^{\al}Z J \;\text{ and }\;
\Psi_n(X_k^{\al}Y)=(-1)^{\al} X_k^{\ov{\al}}.$$
Hence,
$$\begin{array}{rl}
\si_t(A_1^{\ov{\al_1}}\cdots A_r^{\ov{\al_r}})&
\stackrel{\mu_n}{\longrightarrow}
(-1)^{t(\al_1+\cdots+\al_r)} \,\si_t(Z^TA_1^{\al_1}ZJ\cdots Z^TA_r^{\al_r}ZJ)\\
&\stackrel{\theta_n}{\longrightarrow}
(-1)^{t(\al_1+\cdots+\al_r)} \,\si_t(A_1^{\al_1}Y\cdots A_r^{\al_r}Y)\\
&\stackrel{\Psi_n}{\longrightarrow} \si_t(A_1^{\ov{\al_1}}\cdots A_r^{\ov{\al_r}})\qquad\text{and }\\
\end{array}$$
$$\begin{array}{rl}
\si_t(A_1^{\al_1}Y\cdots A_r^{\al_r}Y)&
\stackrel{\Psi_n}{\longrightarrow}
(-1)^{t(\al_1+\cdots+\al_r)} \,\si_t(A_1^{\ov{\al_1}}\cdots A_r^{\ov{\al_r}})\\
&\stackrel{\theta_n\circ\mu_n}{\longrightarrow}
\si_t(A_1^{\al_1}Y\cdots A_r^{\al_r}Y),\\
\end{array}$$
where $\theta_n \circ \mu_n (x)$ stands for $\theta_n(\mu_n(x))$. Hence
\begin{eq}\label{eq_id1}
\Psi_n\circ \theta_n \circ \mu_n \text{ is the identical map on } R^{\Sp(n)} \text{ and }
\end{eq}
\vspace{-0.5cm}
\begin{eq}\label{eq_id2}
\theta_n\circ \mu_n \circ \Psi_n \text{ is the identical map on } I_n. 
\end{eq}%
Lemma~\ref{lemma_iso} is proven.
\end{proof}

%=======================================================================================
%=======================================================================================
%---S5-------------------------------------------------------------------------------
\section{Invariants of quivers}\label{section_quivers}

Consider the following quivers: 

$$
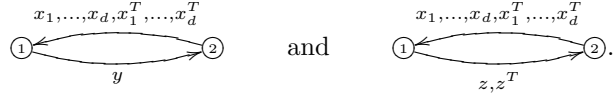
\xymatrix@C=1cm@R=1cm{ %
\vtx{1}\ar@1@/^/@{<-}[rr]^{x_1,\ldots,x_d,x_1^T,\ldots,x_d^T} &&\vtx{2}\ar@1@/^/@{<-}[ll]^{y}\\
}%
\qquad\text{ and } \qquad
\xymatrix@C=1cm@R=1cm{ %
\vtx{1}\ar@1@/^/@{<-}[rr]^{x_1,\ldots,x_d,x_1^T,\ldots,x_d^T} &&\vtx{2}\ar@1@/^/@{<-}[ll]^{z,z^T}\\
}.
$$
Denote the left hand side quiver by $\Q_y$ and the right hand side quiver by $\Q_z$. As in Section~\ref{section_identities}, by abuse of notation some arrows of this quivers are denoted by letters from $\X$. Let $\LA \Q_y\RA$ ($\LA \Q_z\RA$, respectively) be the set of all closed paths in $\Q_y$ ($\Q_z$, respectively). By definition, we set $y^T=-y$. Extend the definition of the lexicographical order on $\X$ as follows: $y$ is less than $x_i,x_i^T$ and $z,z^T$ are both less than $x_i,x_i^T$ ($i>0$). Then we define $\stackrel{T}{}$-involution  and $\sim$-equivalence on $\LA \Q_y\RA$ and $\LA \Q_z\RA$ in the natural way.  Considering $\LA \Q_y\RA$ instead of $\X$, we define $\LA \widetilde{\Q_y}\RA$, $\FF\LA \Q_y\RA$, $\si\LA \Q_y\RA$ and $\si\LA \widetilde{\Q_y}\RA$ similarly to $\EX$, $\FF\X$, $\si\X$ and $\si\EX$, respectively (see Section~\ref{section_identities}). The same notions can also be introduced for $\Q_z$. 

We write $I(\Q_z,n)$ for the algebra of polynomial invariants of mixed representations of dimension vector $(n,n)$ of the quiver
$$
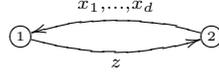
\xymatrix@C=1cm@R=1cm{ %
\vtx{1}\ar@1@/^/@{<-}[rr]^{x_1,\ldots,x_d} &&\vtx{2}\ar@1@/^/@{<-}[ll]^{z}\\
}%
$$
with involution that interchange vertices $1$ and $2$ (see~\cite{Zubkov05I} or~\cite{Lopatin09JAlg} for more details).  Zubkov~\cite{Zubkov05I} established that the algebra $I(\Q_z,n)$ is generated by the following elements:  $\si_t(A_1B_1\cdots A_rB_r)$ for $1\leq t\leq n$, $r>0$, where $A_i\in\{X_1,\ldots,X_d,X_1^T,\ldots,X_d^T\}$ and $B_i\in\{Z,Z^T\}$ for all $i$.  

We set $X_y=Y$, $X_z=Z$, and $X_{z^T}=Z^T$, where matrices $Y$ and $Z$ were defined in Section~\ref{section_isomorphism}. Hence $X_a$ is determined for every $a\in\LA \Q_y\RA$ as well as for $a\in\LA \Q_z\RA$ (see also Section~\ref{section_identities}). Define the surjective homomorphisms of algebras 
$$\widetilde{\Phi}_{y,n}:\si\LA \widetilde{\Q_y}\RA \to I_n \text{ and } \widetilde{\Phi}_{z,n}:\si\LA \widetilde{\Q_z}\RA \to I(\Q_z,n)$$ %
by $\si_t(a) \to \si_t(X_a)$, if $t\leq n$, and $\si_t(a) \to 0$ otherwise. Kernels $K_{y,n}$ and $K_{z,n}$, respectively, of these maps are ideals of relations for $I_n$ and $I(\Q_z,n)$, respectively. 

Assume that $u,v$ are verteces of $\Q_y$. Given an $a\in\FF\LA\Q_y\RA$, we have  $a=\sum_i\al_ia_i$ for some $\al_i\in\FF$ and $a_i\in\LA\Q_y\RA$. If $a_i'=u$ and $a_i''=v$ for all $i$, then we say that $a$ is {\it regular} and write $a'=u$ and $a''=v$.

%----R5.1--------------------------------------------------------------
\begin{remark}\label{remark_QL}
As in Lemma~\ref{lemma_L}, any element of $\si\LA \Q_y\RA$ can be considered as an element of $\si\LA \widetilde{\Q_y}\RA$. Namely, assuming that $a_1,\ldots,a_s$ from the definition of $L$ (see Lemma~\ref{lemma_L}) are regular elements of $\FF\LA \Q_y\RA$ with  $a_1'=\cdots=a_s'=a_1''=\cdots=a_s''$ and $a,bc$ are closed paths in $\Q_y$, we define an ideal $L_y$. Then $\si\LA \widetilde{\Q_y}\RA\simeq \si\LA \Q_y\RA/L_y$ and $L_y$ belongs to the kernel of the homomorphism of algebras $\PhiAbs{y,n}:\si\LA \Q_y\RA \to I_n$ that is defined similarly to $\PhiAbs{n}$. The same remark also holds for $\Q_z$.
\end{remark}
\bigskip

A triple $(a,b,c)$ of regular elements from $\FF\LA\Q_y\RA$ is called {\it $\Q_y$-admissible} if $a'=a''=b'=c''=1$ and $b''=c'=2$.

%----L5.2----------------------------------
\begin{lemma}\label{lemma_relIn} 
Assume $p\neq2$. Then the ideal of relations $K_{y,n}$ for $I_n$ is generated by $\si_{t,r}(a,b,c)$ for $t+2r>n$ ($t,r\geq0$), where $(a,b,c)$ is an admissible triple of $\Q_y$.
\end{lemma}
\begin{proof}
Consider a relation $f\in\si\LA\widetilde{\Q_y}\RA$ for $I_n$. Since $\FF$ is infinite, without loss of generality we can assume that $f$ is multihomogeneous. In particular, each monomial of $f$ has one and the same degree $k\geq0$ in letter $y$. %??? "explain the notion of multidegree!"
Denote by $h\in\si\LA\widetilde{\Q_z}\RA$ the result of substitution $y\to z-z^T$ in $f$. Obvioulsly, $h$ is a relation for $I(\Q_z,n)$.  The general result by Zubkov~\cite{Zubkov05II} implies that the ideal $K_{z,n}$ of relations for $I(\Q_z,n)$ is generated by $\si_{t,r}(a,b,c)$ for $t+2r>n$ and admissible triples $(a,b,c)$ of $\Q_z$. Denote by $l\in\si\LA\widetilde{\Q_y}\RA$ the result of substitution $z\to y$, $z^T\to y^T=-y$ in $h$. Then $l$ belongs to the ideal of $\si\LA\widetilde{\Q_y}\RA$ generated by $\si_{t,r}(a,b,c)$ for $t+2r>n$ and admissible triples $(a,b,c)$ of $\Q_y$. On the other hand, $l=2^k \!f$ and the proof is completed.  
\end{proof}

Recall the notion of a partial linearization $\si_{\un{t};\un{r};\un{s}}$ of $\si_{t,r}$. Assume that $\un{t}\in\NN_0^u$, $\un{r}\in\NN_0^v$, $\un{s}\in\NN_0^w$ satisfy $|\un{r}|=|\un{s}|$ and $a_1,\ldots,a_u,b_1,\ldots,b_v,c_1,\ldots,c_w$ belong to $\FF\X$. Consider  $\si_{t,r}(\al_1 a_1+\cdots + \al_u a_u, \be_1 b_1 + \cdots + \be_v b_v, \ga_1 c_1 + \cdots + \ga_w c_w)\in\si\EX$ as a polynomial in $\al_1,\ldots,\al_u,\be_1,\ldots,\be_v,\ga_1,\ldots,\ga_w$, where $\al_1,\ldots,\ga_w\in\FF$. The coefficient of $\al_1^{t_1}\cdots\al_u^{t_u} \be_1^{r_1}\cdots \be_v^{r_v} \ga_1^{s_1}\cdots \ga_w^{s_w}$ in this polynomial is denoted by
$$\si_{\un{t};\un{r};\un{s}}(\un{a};\un{b};\un{c}) = \si_{\un{t};\un{r};\un{s}}(a_1,\ldots,a_u; b_1,\ldots,b_v; c_1,\ldots,c_w) \in\si\EX.$$ %
Taking $\varrho_{t,r}$ instead of $\si_{t,r}$, we define the partial linearization $\varrho_{\un{t};\un{r};\un{s}}(\un{a};\un{b};\un{c})$ of $\varrho_{t,r}$ in the same way as above. By Lemma~4.2 of~\cite{Lopatin12ART},  there is an explicit formula for $\si_{\un{t};\un{r};\un{s}}$. Namely,
\begin{eq}\label{eq_expl_si}
\si_{\un{t};\un{r};\un{s}}(\un{a};\un{b};\un{c})=\sum (-1)^{\xi} \si_{k_1}(e_1)\cdots \si_{k_q}(e_q).
\end{eq}
% ...more details should be added !!!???
Moreover, we have
\begin{eq}\label{eq_expl_varrho}
\varrho_{\un{t};\un{r};\un{s}}(\un{a};\un{b};\un{c})=\sum (-1)^{t+k_1+\cdots+k_q} \si_{k_1}(e_1)\cdots \si_{k_q}(e_q),
\end{eq}
where $e_1,\ldots,e_q,k_1,\ldots,k_q$ are the same as in formula~\Ref{eq_expl_si}.

Since the field $\FF$ is infinite, Theorem~\ref{theo_relationsSp} is equivalent to the following statement.

%---C5.3----------------------------------------------------------
\begin{cor}\label{cor_relationsSp}
Assume $p\neq2$ and $G=\Sp(n)$. Then the ideal of relations $K_{n}$ for $R^{\Sp(n)}\simeq  \si\EX/K_{n}$ is generated by $\varrho_{\un{t};\un{r};\un{s}}(\un{a};\un{b};\un{c})$ for $|\un{t}| + |\un{r}| + |\un{s}| > n$ ($\un{t}\in\NN_0^u$, $\un{r}\in\NN_0^v$, $\un{s}\in\NN_0^w$), where $a_1,\ldots,a_u,b_1,\ldots,b_v\in\X$ and $c_1,\ldots,c_w\in\X^{\#}$.
\end{cor}
\bigskip

%--------------------------------------------------------------
Now we can prove Corollary~\ref{cor_relationsSp} and, therefore,  Theorem~\ref{theo_relationsSp}.
\begin{proof} 
As above, given an $a\in\FF\X$ and $\al\in\{0,1\}$, we set
$$a^{\al}=\left\{
\begin{array}{rl}
a,&\text{if } \al=0\\
a^T,&\text{if } \al=1\\ 
\end{array}\right..
$$
Consider an isomorpfism of algebras 
$$\Psi:\si\LA\widetilde{\Q_y}\RA\to \si\EX,$$
defined by $\si_t(a_1^{\al_1}y\cdots a_s^{\al_s}y)\to (-1)^{t(\al_1+
\cdots+\al_s)} \si_t(a_1^{\al_1}\cdots a_s^{\al_s})$ for all $a_1,\ldots,a_s\in\{x_1,\ldots,x_d\}$ and $\al_1,\ldots,\al_s\in\{0,1\}$ with 
\begin{eq}\label{eq_sec5_1}
a_1^{\al_1}y\cdots a_s^{\al_s}y\text{ from }\LA\widetilde{\Q_y}\RA. 
\end{eq}%
The last condition is equivalent to the condition that $a_1^{\al_1}\cdots a_s^{\al_s}\in\EX$. Since $y^T=-y$, we can eliminate condition~\Ref{eq_sec5_1} from the definition of $\Psi$ (see also Lemma~\ref{lemma_L} and Remark~\ref{remark_QL}). Obviously, $\Psi\circ \PhiLarge{n}=\PhiLarge{y,n}\circ \Psi_n$. Thus Lemma~\ref{lemma_iso} implies that the ideal of relations $K_n$ for $R^{\Sp(n)}$ is equal to $\Psi(K_{y,n})$. The ideal of relations $K_{y,n}$ for $I_n$ was described in Lemma~\ref{lemma_relIn}. Since $\FF$ is infinite, we can rewrite Lemma~\ref{lemma_relIn} as follows: the ideal $K_{y,n}$ is generated by $\si_{\un{t};\un{r};\un{s}}(\un{a},\un{b},\un{c})$ for $|\un{t}|+|\un{r}|+|\un{s}|>n$ ($t,r\geq0$), where $(a_i,b_j,c_k)$ is a $\Q_y$-admissible triple of {\it paths} in $\Q_y$ for all $i,j,k$.

Consider a $\Q_y$-admissible triple $(a,b,c)$ of paths in $\Q_y$, i.e., $a=a_1^{\al_1}y\cdots a_k^{\al_k}y$, $b=b_1^{\be_1}yb_2^{\be_2}\cdots y b_l^{\be_l}$, $c=c_1^{\ga_1}y\cdots c_m^{\ga_m}y$ for $k,l>0$, $m\geq0$, $a_1,\ldots,a_k$, $b_1,\ldots,b_l$, $c_1,\ldots,c_m\in\{x_1,\ldots,x_d\}$ and $\al_1,\ldots,\ga_m\in\{0,1\}$. 

Denote by $\LA X,y\RA^{\#}$ the monoid (with unity),  freely generated by $x_1,\ldots,x_d$, $x_1^T,\ldots,x_d^T$ and $y$, where $y^T=-y$. To compute $\Psi(\si_{t,r}(a,b,c))$, we introduce a homomorphism of monoids $\Psi_0:\LA X,y\RA^{\#} \to \X^{\#}$ defined by $x_i\to x_i$, $x_i^T\to -x_i^T$, $y\to 1$.  Note that
\begin{enumerate}
\item[$\bullet$] $\Psi_0(a)=(-1)^{\al_1+\cdots+\al_k}\, a_1^{\al_1}\cdots a_k^{\al_k}$ and $\Psi_0(a^T)=\Psi_0(a)^T$;

\item[$\bullet$] $\Psi_0(b)=(-1)^{\be_1+\cdots+\be_l}\, b_1^{\be_1}\cdots b_l^{\be_l}$ and $\Psi_0(b^T)=-\Psi_0(b)^T$;

\item[$\bullet$] $\Psi_0(c)=(-1)^{\ga_1+\cdots+\ga_m}\, c_1^{\ga_1}\cdots c_m^{\ga_m}$ and $\Psi_0(c^T)=-\Psi_0(c)^T$.
\end{enumerate}
Since condition~\Ref{eq_sec5_1} may not hold in the definition of $\Psi$, we have $\Psi(\si_t(e))=\si_t(\Psi_0(e))$ for all $e\in\LA \Q_y\RA$. Therefore, the definition of $\si_{t,r}$ implies that
$$\Psi(\si_{t,r}(a,b,c)) = \varrho_{t,r}(\Psi_0(a),\Psi_0(b),\Psi_0(c)).$$
Repeating the above reasoning and applying formulas~\Ref{eq_expl_si} and \Ref{eq_expl_varrho}, we obtain that 
$$\Psi(\si_{\un{t};\un{r};\un{s}}(\un{a};\un{b};\un{c})) =  \varrho_{\un{t};\un{r};\un{s}}(\Psi_0(a_1),\ldots,\Psi_0(a_u); \Psi_0(b_1),\ldots,\Psi_0(b_v); \Psi_0(c_1),\ldots,\Psi_0(c_w))$$
for all $a_1,\ldots,c_w\in\X^{\#}$. The above given reformulation of Lemma~\ref{lemma_relIn} completes the proof.
\end{proof}

%=======================================================================================
%=======================================================================================
%---S6-------------------------------------------------------------------------------
\section{Failure of the Weyl property}

%---Th6.1-----------------------------------------------------------------
\begin{theo}\label{theo_weyl}
Assume that the characteristic of $\FF$ is not two and $G$ is $O(n)$ or $\Sp(n)$. Then $\tr(X_1\cdots X_d)$ is not decomposable in $R^G$ for all $d>0$ if and only if $0<p\leq n$. In particular, property~\Ref{eq_weyl} holds for $R^G$. 
\end{theo}

%--------------------------------------------------------------------------------
\section*{Acknowledgements}
This paper was supported by FAPESP No.~2011/51047-1 and RFFI 13-01-00239A. The author is grateful for support.

%---bib--------------------------------------------------------------------------------

\end{document}